\documentclass[12pt]{amsart}
\usepackage{enumerate,enumitem,amsmath,amssymb, mathrsfs,stmaryrd}

\usepackage{latexsym, amssymb,enumerate}
\usepackage{appendix}

\newcommand{\bpf}{\begin{proof}}
\newcommand{\epf}{\end{proof}}
\newcommand{\beq}{\begin{equation}}
\newcommand{\eeq}{\end{equation}}
\newcommand{\beqn}{\begin{eqnarray*}}
\newcommand{\eeqn}{\end{eqnarray*}}

\newcommand\diam{\mathop{\rm diam}\nolimits}
\newcommand\tr{\mathop{\rm tr}\nolimits}

\def\circledwedge{\setbox0=\hbox{$\bigcirc$}\relax \mathbin {\hbox
to0pt{\raise.5pt\hbox to\wd0{\hfil $\wedge$\hfil}\hss}\box0 }}

\newtheorem{prop}{Proposition}[section]
\newtheorem{theo}[prop]{Theorem}
\newtheorem{lemm}[prop]{Lemma}
\newtheorem{coro}[prop]{Corollary}
\newtheorem{rema}[prop]{Remark}

\def\begeq{\begin{equation}}
\def\endeq{\end{equation}}
\def\p{\partial}

\def\tr{{\rm tr}}

\def \ds{\displaystyle}

\def\div{{\rm div\,}}

\def\odot{\setbox0=\hbox{$\bigcirc$}\relax \mathbin {\hbox to0pt{\raise.5pt\hbox to\wd0{\hfil $\wedge$\hfil}\hss}\box0 }}

\numberwithin{equation} {section}
\def\tilde{\widetilde}

\begin{document}

\title[Conformally compact Einstein manifolds] {Compactness of conformally compact Einstein $4$-manifolds II}

\author{Sun-Yung A. Chang}
\address{Department of Mathematics, Princeton University, Princeton, NJ 08544, USA}
\email{chang@math.princeton.edu}

\author{Yuxin Ge}
\address{IMT,
Universit\'e Toulouse 3 \\118, route de Narbonne
31062 Toulouse, France}
\email{yge@math.univ-toulouse.fr}

\author{Jie Qing}
\address{Mathematics Department, UCSC, 1156 High Street, Santa Cruz, CA 95064 USA }
\email{qing@ucsc.edu}

\begin{abstract} In this paper, we establish a number of compactness results for some class of conformally compact Einstein $4$-manifolds. 
In the first part of the paper, we improve the earlier compactness results obtained by Chang-Ge  \cite{CG}. In the second part of the paper, as applications, we
derive some further compactness results under the perturbation condition when the $L^2$ norm of the Weyl curvature is small. 
We also derive the global uniqueness of conformally compact Einstein metrics on the 4-Ball constructed in the earlier work of Graham-Lee \cite{GL}.
\end{abstract}

\thanks{Research of Chang is supported in part by NSF grant MPS-1509505}
\thanks{Research of Qing is supported in part by NSF grant MPS-1608782}

\subjclass[2000]{}


\maketitle

\section{Introduction}\label{introduction and statement of results}

\subsection{Statement of Improved results}

Let $X^4$ be a smooth 4-manifold with boundary $\partial X$.
A smooth conformally compact metric $g^+$ on $X$ is a Riemannian metric such that $g = r^2 g^+$ extends smoothly to the 
closure $\overline{X}$ for some defining function $r$ of the boundary $\partial X$ in $X$. A defining function $r$ is a smooth 
nonnegative function on the closure $\overline{X}$ such that $\partial X = \{r=0\}$ and $d r \neq 0$ on $\partial X$. A 
conformally compact metric $g^+$ on $X$ is said to be conformally compact Einstein (CCE) if, in addition, 
$$
\operatorname{Ric} [g^+] = - n g^+.
$$
The most significant feature of CCE manifolds $(X, \ g^+)$ is that the metric $g^+$ is ``canonically" associated with the conformal
structure $[\hat g]$ on the boundary at infinity $\partial X$, where $\hat g = g|_{T\partial X}$. $(\partial X, \ [\hat g])$ is called the conformal infinity of a conformally 
compact manifold $(X, \ g^+)$. It is of great interest in both mathematics and theoretic physics to understand the correspondences between 
conformally compact Einstein manifolds $(X, \ g^+)$ and their conformal infinities $(\partial X, \ [\hat g])$, especially due to the AdS/CFT
correspondence in theoretic physics (cf. Maldacena \cite{Mald-1}, \cite{Mald-2} and  Witten  \cite{Wi}). \\ 

The project we work on is to address the compactness issue of given a sequence
of CCE manifolds $(X^4, M^3, \{g_i^{+}\})$ with $M = \partial X$ and $\{ g_i \} = \{ r_i^2 g_i^{+} \} $ a sequence of compactified metrics, denote
${\hat g}_i = g_i |_M$,  assume $\{{\hat g}_i\}$ forms a compact family of metrics in $M$, when is it true that some representatives ${\bar g_i} \in  [g_i]$
with $\{ {\bar g}_i |_M\} = \{{\hat g}_i\} $ also forms a compact family of metrics in $X$?
We remark that, for a CCE manifold, given any conformal infinity, a special defining function which we call geodesic defining function $r$ exists 
so that $ | \nabla_{\bar g} r | \equiv 1 $ in an asymptotic neighbor $M \times [0, \epsilon)$ of $M$.     
We also remark that the eventual goal to study the compactness problem is to show existence of conformal filling in for some classes of
Riemannian manifolds as conformal infinity.
\\

One of the difficulty to address the compactness problem is due to the existence of some ``non-local" term. To see this, we look at  
 the asymptotic behavior of the compactified metric $g$ of CCE manifold $(X^{n+1} ,M^n, g^{+})$ with conformal infinity $(M^n, \hat g )$ (\cite{G00}, \cite{FG12}) which in the special case when $n=3$ takes the form
$$  
 g:=r^2 g^{+} = h + g^{(2)} {r}^2 + g^{(3)} r^3 + g^{(4)} {r}^4 + \cdot \cdot \cdot \cdot 
$$
on an asymptotic neighborhood of $M \times (0, \epsilon)$, where $r$ denotes the geodesic defining function of $g$. It turns out
$g^{(2)} = - \frac{1}{2}A_{\hat g} $, where $A_{\hat g} := \frac {1}{n-2} ( Ric_{\hat g} - \frac {1}{2(n-1)} R_{\hat g}) $ denotes the Schouten tensor, Ric the Ricci tensor and R 
the scalar curvature respectively for the metric $\hat g$. Thus $g^{(2)}$ is  determined by $\hat g$ (we call such terms local terms), $Tr_{\hat g}  g^{(3)} =0 $, while 
$$  
g^{(3)} _{\alpha, \beta} = - \frac {1}{3} {\frac {\partial}{\partial n }{(Ric_{g})}_{\alpha, \beta} } $$
where $\alpha, \beta$ denote the tangential coordinate on $M$, is a non-local term which is not determined by the boundary metric $\hat g$.   
We remark that $\hat g $ together with $g^{(3)}$ determine the asymptotic behavior of $g$ (\cite {FG12}, \cite{Biquard}). \\

In an earlier work of Chang-Ge \cite{CG}, for a CCE manifold $(X^4, M^3, g^{+})$, we introduce the notion of
2-tensor $S$ which on a 4-manifold $(X^4, \ g)$ with totally geodesic boundary takes the form:
$$
S[g] := \frac12 \partial_{\bf n} \operatorname{Ric} - \frac{1}{12} \partial_{\bf n} R\, g,
$$
where ${\bf n}$ is the outward unit normal of the 
boundary under the metric $g$. The 2-tensor $S$ is conformally invariant in the sense that
$$
S[r^2 g]=r^{-1} S[g].
$$
The connection of the $S$ tensor to that of $g^{(3)}$ is that (see  (2.7), Remark 2.1 in \cite{CG}) 
\begin{equation}
\label{g3}
S_{\alpha, \beta} = - \frac{3}{2} g^{(3)}_{\alpha, \beta}. 
\end{equation} \\

In  \cite{CG}, we have also considered a special choice of compactification $g^{*}$,  which we named the Fefferman-Graham's (FG) compactification, defined
by solving the PDE:
\begin{equation}
\label{fg}
- \Delta_{g^{+}} w = 3  \,\,\,\, on  \,\,\, X^4.
\end{equation}
$g^{*} : = e^{2w} g^+ $ with $g^{*} |_M = g^Y$, the Yamabe metric on the conformal infinity of $(X^4, g^+)$. \\ 

We now state the first result of this paper.

\begin{theo} 
\label{maintheorem1}
Suppose that $X$ is a smooth oriented 4-manifold with boundary $\partial X=\mathbb{S}^3$. Let $\{g_i^+\}$ be a set of conformally compact Einstein metrics on $X$. Assume the following conditions:

\begin{enumerate}
\item \label{Con:1}The set $\{\hat g_i\}$ of Yamabe metrics that represent the conformal infinities lies in a given set $\mathcal{C}$ 
of metrics that is of positive Yamabe type and compact in $C^{k,\alpha}$ Cheeger-Gromov topology with $k \ge 3$ and with some $\alpha\in (0,1)$.
\item \label{Con:2} The FG  compactifications $\{ g^{*}_i = \rho_i^2 g^+_i\}$ associated with the Yamabe representatives $\{\hat g_i\}$ 
on the boundary satisfies:
$$
\lim_{r\to 0}\sup_i\sup_{x\in \partial X}\oint_{B(x,r)} |S_i|[g^{*}_i] d\text{vol}[\hat g_i] =0
$$
\item \label{Con:3} $
H_2(X, \mathbb{Z})=0$.
\end{enumerate}
Then, the set $\{g^{*}_i\}$ of FG compactifications (after diffeomorphisms that fix the boundary) 
forms a compact family in the $C^{k,\alpha}$ Cheeger-Gromov topology.
\end{theo} 
\vskip .1in
%

We now explain the connection of the $S$ tensor to other scalar curvature invariants for the metric $g^{*}$, which plays a key role in the results in \cite{CG} and in this
paper.\\

On a 4-manifold $(X^4, \  g)$, a 4-th order curvature called the   
$Q$-curvature is defined as:
\beq\label{Q-curvature}
Q[g] :=-\frac16 \triangle R-\frac12|Ric|^2+\frac16 R^2.
\eeq
$Q$ curvature is naturally associated with a $4$th-order differential operator:  
$$
P[g] :=( \triangle)^2-\div[(\frac23 R g-2 Ric)\nabla]
$$
called Paneitz operator \cite{Pa, BCY}. We remark that Paneitz operator is a special case of the family of GJMS (\cite{GJMS92}) operators.
The relation of the pair $\{Q, P\}$ in 4 dimensions is like that of  the well known pair $\{K, -\Delta\}$ in 2 dimensions, where $K$ denotes the Gaussian curvature:
 $$ 
- \Delta [g] + K[g] = K{e^{2w} g} e^{2w} \,\,\,\, {\text on \,\,\, X^2},
$$
$$
P[g] w + Q[g]=Q[e^{2w}g] e^{4w} \,\,\,\, {\text on \,\,\, X^4}
$$
for conformal changes of the metric.  For a 4-manifold $(X^4, \ g)$  with boundary, in the earlier works of Chang-Qing \cite{CQ1, CQ2}, 
in connection with the 4th order $Q$ curvature, a 3rd order  
"non-local" boundary curvature $T$ was introduced  on $ \p X$ to study the boundary behavior of $g$. The relation between the pair $(Q, T)$ is a 
generalization of that of the Dirichlet-Neumann pair $(- \Delta, \partial_{\bf n})$. The expression of $T$ curvature is in general complicated, but in the special case 
when $g$ is totally geodesic,  the expression $T$ take the simple form:
\beq\label{T-curvature}
T[g]:=\frac{1}{12}\frac{\p R}{\p {\bf n}}.
\eeq 
\\

We now state the second result of our paper.

\begin{theo} \label{maintheorem2}
Suppose that $X$ is a smooth oriented  4-manifold with boundary $\partial X=\mathbb{S}^3$. Let $\{g_i^+\}$ be a set of conformally 
compact Einstein metrics on $X$. Assume the same conditions \eqref{Con:1} and \eqref{Con:3} as in Theorem \ref{maintheorem1} and 
\begin{enumerate}
\item[]($2^\prime$) For the associated Fefferman-Graham's compactifications $\{g_i^{*} = e^{2w_i}  g^+_i\}$ with the Yamabe representatives $\{\hat g_i\}$ on the
boundary,  
$$
\liminf_{r\to 0}\inf_i\inf_{x\in \partial X}\oint_{B(x,r)}T[g_i^{*} ] d\text{vol}[\hat g_i] \ge 0.
$$
\end{enumerate}
Then, the set $\{g_i^{*} \}$ is compact in $C^{k,\alpha}$ norm 
up to diffeomorphisms that fix the boundary, provided $k \ge 7$.
\end{theo} 
\vskip .1in
In the earlier work of Chang-Ge \cite{CG}, both Theorem \ref{maintheorem1} and  Theorem \ref{maintheorem2}  were established under the extra conditions 
on the uniform bounds of the first and second Yamabe constants for manifolds with boundary.  The conditions are:
\\

\begin{enumerate}
\item[](4) There exists some positive constant $C_5>0$ such that the first Yamabe constant for the compactified metric $g_i:=\rho_i^2 g_i^+$ is bounded uniformly from below by $C_2$
i.e.
$$
Y(X, \partial X,  [g_i]) :=\inf_{U\in C^1}\frac{\ds\int_X (|\nabla U|^2 +\frac16 R[g_i] U^2) d\text{vol}[g_i]}{\ds\left(\int_X U^4 d\text{vol}[g_i] \right)^{\frac{1}{2}}}  \ge C_5,
$$
where $R[g_i]$ is the scalar curvature of $g_i$.
\item[](5)  There exists some positive constant $C_6>0$ such that the second Yamabe constant for the metric $g_i$ is bounded uniformly from below by $C_6$, i.e. 
$$
Y_b(X, \partial X, [g_i]) :=\inf_{U\in C^1}\frac{\ds\int_X (|\nabla U|^2 +\frac16 R[g_i] U^2) d\text{vol}[g_i]}{\ds\left(\oint_{\partial X} U^3 d\text{vol}[\hat g_i] \right)^{\frac{2}{3}}}\ge C_6,
$$
where $\hat g_i = g_i|_{T\partial X}$.
\end{enumerate}
We remark Condition \eqref{Con:1} in the earlier work \cite[Theorem 1.1]{CG} is stated slightly weaker, that is, the Yamabe 
constant $Y(\partial X, [\hat g_i])$ is assumed to be non-negative.
\\

In the current paper we managed to drop both conditions (4) and (5) from the statements of Theorems \ref{maintheorem1} and \ref{maintheorem2}, 
which is accomplished  by applying another round of blow-up analysis to reduce the situation to the earlier version of the theorems in \cite{CG}. We will present the proof
in section 3 of this paper.   Once the curvature of metric $g_i^*$ is bounded, we will prove the diameter is uniformly bounded in section 4 by  some new arguments.



\subsection{Statement of new results}

Due to the nature of the problem in the CCE setting, natural conditions to imply the compactness of the solutions should be conformally invariant conditions, condition (1) 
in the statements of Theorems \ref{maintheorem1} and \ref{maintheorem2} is conformally invariant, condition (3) is topologically invariant hence a conformal invariant, but unfortunately the condition (2) in Theorem \ref{maintheorem1}  
and condition (2') in Theorem \ref{maintheorem2} are not. It is in this direction we now have new results where compactness is reached under 
some conformally invariant conditions; as a consequence we also establish some "uniqueness" result of conformal filling in for a special class of CCE manifolds with given conformal infinity. \\ 

\begin{theo} \label{maintheorem3}
Suppose that $X$ is a smooth oriented  4-manifold with boundary $\partial X=\mathbb{S}^3$. Let $\{g_i^+\}$ be a set of conformally 
compact Einstein metrics on $X$. Assume the same condition \eqref{Con:1} in Theorem \ref{maintheorem1}. Then there is $\delta_0>0$
such that if either 
\begin{enumerate}
\item[]($2^{\prime\prime}$) \quad  \quad $\int_{X^4} (|W|^2dvol)[g^+_i] < \delta_0$, \quad \quad \quad 
or \\ 
 \item[]($2^{\prime\prime\prime}) $  \quad \quad  $Y(\partial X, [\hat g_i]) \geq Y(\mathbb{S}^n, [g_{\mathbb{S}}]) - \delta_0, $ 
\end{enumerate}
then the set $\{g_i^*\}$ of the FG compactifications (after diffeomorphisms that fix the boundary) 
is compact in $C^{k,\alpha}$ Cheeger-Gromov topology for some $k \geq 3$.
 \\
 
In fact, for $\epsilon >0$ and $\alpha\in (0, 1)$, there is $\delta >0$, if $g^+$ is a conformally compact Einstein metric on $X^4$ with the 
conformal infinity $(\partial X, [\hat g])$ and $g^*$ is the FG compactification associated with the Yamabe representative  
that belongs the set $\mathcal{C}$ in (1) in Theorem \ref{maintheorem1}, and if ($2^{\prime\prime}$) (or ($2^{\prime\prime\prime}$)) holds for some $\delta\leq  \delta_0 \, $, then 
there is a diffeomorphism 
$$
\phi_i: \bar X^4\to \bar B^4 \text{ and }\phi_i=Id: \partial B^4=\mathbb{S}^3\to \partial X^4=\mathbb{S}^3
$$ 
satisfying
\begin{equation}\label{Equ:close}
\|\phi_i^*g_i^* - g_{FG}\|_{C^{k, \alpha}(\bar B^4)} < \epsilon
\end{equation}
where $g_{FG}$ is the FG compactification of the hyperbolic metric associated with a round metric on $\mathbb{S}^3$. 
\end{theo}
 
We will now relate the condition ($2^{\prime\prime}$)  in Theorem \ref{maintheorem3} to some other natural geometric conformal invariant, namely the "renormalized volume"
in the CCE setting. Although the renormalized volume can be defined on CCE manifolds $(X^{n+1}, \p X, g^{+}) $ for any dimension $n$, we will here mainly recall 
some basic facts on CCE manifolds $(X^4, \p X, g^+)$ when $n=3$. \\ 

The concept of ``renormalized volume" in the CCE setting was introduced by 
Maldacena \cite{Mald-1} (see also the works of Witten \cite{Wi}, Henningson-Skenderis \cite{HeK98} and Graham \cite{G00}).
On CCE manifolds $(X^{n+1}, M^n, g^{+})$ with geodesic defining function $r$, \newline 
For $n$ odd,
$$ \aligned \text{Vol}_{g^{+}}(\{r>
\epsilon\}) & = c_0 \epsilon^{-n} + c_2\epsilon^{-n+2} + \cdots \cdot \cdot  \\
& + c_{n-1} \epsilon^{-1} + V + o(1).\\
\endaligned $$
We call the zero order term $V$ the renormalized volume. It turns out  for $n$ odd, $V$ is independent of $g^{+} \in [g^+]$, and hence are conformal invariants.

We now recall Gauss-Bonnet-Chern formula on compact 4-manifolds $(X^4, \p X, g)$ with totally geodesic boundary.
\begin{equation} 
\label{CGB0}
8 \pi^2 \chi(X)  =  \frac 14\int_X( |W|^2 d\text{vol})[g] + 4\int_X \sigma_2(A_g)dvol[g], 
\end{equation}
where $\sigma_2(A_g) := \frac{1}{4} ( \frac{1}{6} R_g ^2 - \frac{1}{2} |Ric|_g^2) $ is the second elementary symmetric function of the Schouten tensor $A_g$. 
We also recall an earlier result:

\begin{prop}(M. Anderson \cite{An01}, Chang-Qing-Yang \cite{CQY06}, \cite{CQY08})
\label{reno} \\
On conformal compact Einstein manifold $(X^4, M^3, g^{+})$, we have
$$ V = \frac {2}{3} \int_{X^4} \sigma_2 (A_g) dv_g $$
for any compactified metric $g$ with totally geodesic boundary. Thus
\begin{equation}
8 \pi^2 \chi (X^4, M^3) = \frac{1}{4}  \int |W|^2_g dv_g + 6 V.
\end{equation}
\end{prop}

We briefly recall the proof of above Proposition in Chang-Qing-Yang \cite{CQY06}, as this is the crucial point that leads us to
adopt the Fefferman-Graham's compactification to study the compactness problem of CCE manifolds.

\begin{proof}[Sketch proof of Proposition \ref{reno}] 
\begin{lemm} (Fefferman-Graham \cite{FG})\label{fgmetric}
Suppose $(X^{4}, {\p X}, g^{+})$ is conformally compact Einstein with conformal infinity $({\p X}, [\hat g])$,  fix $\hat g \in [\hat g ]$ and $r$ its
corresponding geodesic defining function.  Consider the solution  $w$ to (\ref{fg}),
then $ w $ has the asymptotic behavior 
$$ w= log \, r+ A + B r ^3 $$ near ${\p X}$, 
where $A, B$ are functions even in $r$, $A|_{\p X}=0$, then
$$ V = \int_{\p X} B|_{\p X}. $$
\end{lemm}
 
\begin{lemm} (Chang-Qing-Yang \cite{CQY06})
\label{fgm2} 
With the same notation as in Lemma \ref{fgmetric}, 
Consider the metric $g^{*} = g_w = e^{2w} g^+ $, then $g^*$ is totally geodesic on boundary with 
(1)  $Q_{g^{*}} \equiv 0 ,$  
(2)  $ B|_{\p X} = \frac{1}{36} \frac{\partial}{\partial n} R_{g^*} = \frac {1}{3} T_{g^{*}} .$
\end{lemm}

\begin{proof}
Recall we have $g^+$ is Einstein with $Ric_{g^+}=-3g^+$, thus $$P_{g^+} = (- \Delta_{g^+}) \circ (- \Delta_{g^+} -2) $$
and  $ Q_{g^+} = 6.$
Therefore
$$ P_{g^+} w + Q_{g^+} = 0 =  e^{2w} Q_{g^{*}}. $$
Assertion (2) follows from a straight forward computation using the scalar curvature equation and the asymptotic behavior of $w$.
\end{proof} 

Applying Lemmas \ref{fgmetric} and \ref{fgm2}, we get 
$$
\aligned 6 V  & = 6 \oint_{\p X} B|_{\p X} d\sigma_{h} 
= \frac {1}{6} \oint_{\p X} \frac{\partial}{\partial n} R_{g^*} d\sigma_h \\
 & = \int_X Q_{g^{*}} + 2 \oint_{\p X}T_{g^{*}}   = 4 \int_{X} \sigma_2 (A_{g^{*}}) dv_{g^{*}}. \\
\endaligned
$$

For any other compactified metric $g$ with totally geodesic boundary, $ \int_{X^4} \sigma_2 (g) dv_g $
is a conformal invariant, and $V$ is a conformal invariant, thus
the result holds once for $g^*$, holds for any such $g$ in the same conformal class, which establishes Proposition \ref{reno}.
\end{proof}

We also recall some well known fact that (cf. \cite{CQY04, CQY06}).
for a conformally compact Einstein 4-manifold with the conformal infinity of positive Yamabe type, 
\begin{equation}
V(X^4, g^+) \leq V(\mathbb{H}^4, g_{\mathbb{H}}) = \frac {4\pi^2}3
\end{equation}
where the equality holds if and only if $(X^4, g^+)$ is isometric to $(\mathbb{H}^4, g_{\mathbb{H}})$.

We now restrict our attention to class of CCE manifolds $(B^4, S^3, g^+)$, in this class, for the model case when $g^+ = g_{\mathbb{H}}$,  formulas for the specific FG $g^{*}$ metric 
can be computed straight forwardly.

\begin{lemm}
On $(B^4, S^3, g_{\mathbb H})$, 
$$ g^{*} = e^{ (1- |x|^2)}  |dx|^2 \,\,\, on \,\,\, B^4$$  
$$ Q_{g^{*}} \equiv 0, \,\,\,\,\,  T_{g^{*}}  \equiv 2    \,\,\, on \,\,\, S^3 $$
$$ (g^{*})^{(3)} \equiv 0 $$
and
$$ \int_{B^4} \sigma_2 (A_{g*}) dv_{g^{*}}  =  2\, \pi^2. $$
\end{lemm} 
 
On $(B^4, S^3, g )$, 
for a compact metric $g$  with totally geodesic boundary, Gauss-Bonnet-Chern
formula takes the form:
$$
8 \pi^2 \chi (B^4, S^3) = 8 \pi^2 = \int_{B^4} ( \frac {1}{4} |W|_g^2 +  4 \sigma_2 (A_g)) dv_g,  
$$
 
Thus we reached the following corollary of Theorem \ref{maintheorem3}:
 
\begin{coro} 
\label{application1}
Let  $\{X=B^4,M=\partial X=S^3, g^+\}$ be a  $4$-dimensional oriented CCE on $X$ with boundary $\partial X$.  
Assume the boundary Yamabe metric $h= h^Y$ in the conformal infinity of positive type and $Y(S^3, [h])> c_1$ for some fixed $c_1>0$ and  $h$ is compact 
in $C^{k,\alpha}$ norm with $k\geq 3$. Let $g^*$ be the corresponding FG compactification.
Then the following properties are equivalent: 
\begin{enumerate}
\item There exists  some small positive number $\varepsilon>0$ such that
$$
\int_X \sigma(A_{g*}) dv_{g^{*}}  \geq 2\pi^2- \varepsilon.
$$ 
\item There exists some small positive  number $\varepsilon>0$ such that
$$
\int_X |W|_{g^+} ^2 dv_{g^+}  \le  4 \varepsilon .
$$
\item There exists  some small positive  number $\varepsilon_1>0$ such that
$$
Y(S^3,[g_{c}])\, \geq \,  Y(S^3,[h])> Y(S^3,[g_{c}])- \varepsilon_1
$$
where $g_{c}$ is the standard metric on $S^3$.
\item There exists some small positive number $\varepsilon_2>0$ such that for all metrics $g^{*}$ with boundary metric $h$ same
volume as the standard metric $g_c$ on $S^3$, we have
$$ 
T( g^{*}) \geq 2 - \varepsilon_2.
$$
\item There exists  some small positive  number $\varepsilon_3>0$ such that
$$
|(g^{*})^{(3)}| \leq \varepsilon_3.
$$
\end{enumerate}
Where all the $\varepsilon_{i} $ (i = 1,2,3) tends to zero when $\varepsilon$ tends to zero and vice versa for each $i$. 
\end{coro}

As an application of Theorem \ref{maintheorem3}, we are able to establish the global uniqueness for the conformally compact Einstein 
metrics on $\mathbb{B}^4$ with prescribed conformal infinities that very close to the conformal round 3-sphere (cf. \cite{GL, Lee1, LQS}). 
Namely,

\begin{theo} \label{uniqueness} For a given conformal 3-sphere $(\mathbb{S}^3, [\hat g])$  that is sufficiently close to the round one in $C^{3,\alpha}$ topology with $\alpha\in (0,1)$, 
there is exactly one conformally compact Einstein metric $g^+$ on $\mathbb{B}^4$ whose conformal infinity is the prescribed 
conformal 3-sphere $(\mathbb{S}^3, [\hat g])$.
\end{theo}
\begin{rema}\label{Rmk:section1} We remark
\begin{itemize}
\item First, in Theorems \ref{maintheorem1} and \ref{maintheorem2}, if we assume more assumptions on the topology  in (\ref{Con:3}) $H_1(X, \mathbb{Z})= 
H_2(X, \mathbb{Z})=0$, we can drop the assumption on the boundary $\partial X=\mathbb{S}^3$ (see \cite{CG}).
\item Second,  in Theorems \ref{maintheorem3}, we do not need the boundary condition $\partial X=\mathbb{S}^3$ for the compactness result in the first part; for the second part, when the boundary condition $\partial X=\mathbb{S}^3$ holds and $\delta_0$ is small enough, we have the topology of $X$ is the ball $\mathbb{B}^4$. Therefore,  Theorem \ref{uniqueness} holds even without the assumption $X=\mathbb{B}^4$.
\item Third, in Theorems \ref{maintheorem1} , \ref{maintheorem2} and \ref{maintheorem3}, we have the compactness results in $C^{2,\beta}$ topology for any $\beta\in (0,1)$ if we assume the compactness (or more generally boundedness) of the boundary metric $\hat{g}$ in $C^3$ topology (or even in $C^{2,\beta}$ topology). Hence, we could expect the uniqueness result in Theorem \ref{uniqueness} holds in the $C^3$ topology   (or even in $C^{2,\beta}$ topology) (thanks to \cite[Theorem 8.29]{GT}).
\end{itemize}
\end{rema}

The rest of this paper is organized as follows: In Section \ref{Sect:prelim}, we recall some basic ingredients in the proofs and list some of their key 
properties. In Section \ref{Sect:injrad}, we prove the injectivity radius estimates which is the major technical step in the
blow-up analysis in Riemannian geometry. In Section \ref{Sect:compactness}, we establish various compactness results for Fefferman-Graham's compactifications and prove Theorems \ref{maintheorem1}, \ref{maintheorem2} and \ref{maintheorem3}. Finally, in Section \ref{Sect:uniqueness}, we prove 
Theorem \ref{uniqueness} to obtain the global uniqueness for the conformally compact Einstein metrics on $\mathbb{B}^4$
constructed earlier in \cite{GL, Lee1}.\\

{\bf Acknowledge.} The authors thank Jason Lotay, Joel Fine and Felix Schulze for their comments of the first version of the draft. In particular, they point out a mistake in Lemma \ref{bdy-injrad} in that version.
 

\section{Preliminaries}\label{Sect:prelim}

\subsection{Fefferman-Graham's compactifications}\label{Subsect:rho}
Suppose that $X$ is a smooth 4-manifold with boundary $\partial X$ and $g^+$ is a conformally compact 
Einstein metric on $X$. Let $g^*=\rho^{2}g^+$ be the Fefferman-Graham's compactification, that is, $w:=\log \rho$ satisfies the equation (\ref{fg}). The function $\rho$ was first introduced in \cite{FG} then in  \cite{CQY06} to study the renormalized volume.
We recall some basic calculations for curvatures under conformal changes. Write 
$g^+ = r^{-2} g$ for some defining function $r$ and calculate
$$
Ric[g^+] = Ric[g] + 2r^{-1} \nabla^2 r + (r^{-1} \triangle r - 3r^{-2} |\nabla r|^2) g.
$$
Then one has
$$
R[g^+] = r^2 (R[g] + 6r^{-1} \triangle r -12r^{-2} |\nabla r |^2).
$$
Here the covariant derivatives is calculated with respect to the metric $g$ (or Fefferman-Graham's compactification $g^*$ in the following). Therefore, for a Fefferman-Graham's compactification $g^*$ of a conformally compact Einstein metric $g^+$, one has 
\beq \label{relation3}
R[g^*] = 6\rho^{-2}(1-  |\nabla\rho|^2),
\eeq
which in turn gives
\beq
\label{relation5}
Ric[g^*]=- 2\rho^{-1} \nabla^2 \rho
\eeq
and
\beq
\label{relation4}
R[g^*] = - 2\rho^{-1} \triangle \rho.
\eeq
Now we recall 

\begin{lemm} (\cite{casechang}, \cite[Lemma 4.2]{CG})
\label{lem4.1}
Suppose that $X$ is a smooth 4-manifold with boundary $\partial X$ and $g^+$ is a conformally compact Einstein metric on 
$X$ with the conformal infinity $(\p X, [\hat g])$ of nonnegative Yamabe type. Let $g^*=\rho^2g^+$ be the Fefferman-Graham's 
compactification associated with the Yamabe metric $\hat g$ of conformal infinity. Then the scalar curvature $R[g^*]$ is positive in 
$X$. In particular, 
\beq\label{estimate-rho}
\|\nabla\rho \| [g^*] \le 1.
\eeq
\end{lemm}

\subsection{Elliptic estimates for Bach-flat and Q-flat metrics}
Next we recall from \cite{CG} the $\varepsilon$-estimates for Fefferman-Graham's compactifications $g^*$ of 
conformally compact Einstein metrics $g^+$. We will continue to use 
the 2-tensor $S$ when deriving estimates for Fefferman-Graham's compactifications, which are Bach-flat and $Q$-flat metrics. 
We first recall  Bach equations in 4 dimensions for Bach-flat metrics:
\beq
\label{Bach-eq-1}
\Delta A_{ij} - \frac 16 R_{,i j} + R_{ikjl}A^{kl} - R_{ik}A^k_{\ j} + \frac 12 W_{ikjl}A^{kl} = 0,
\eeq
where 
$$
A_{ij} = \frac 12 (R_{ij} - \frac 16 R\, g_{ij})
$$ 
is the Schouten tensor, $R_{ikjl}$ and $W_{ikjl}$ are Riemann and Weyl curvature tensors respectively, and 
$Q$-flat equation:
\beq \label{Q-eq}
\Delta R = 3|\operatorname{Ric}|^2 - R^2
\eeq
the light of \eqref{Q-curvature}. One may use Bach equations coupled with $Q$-flat equation to derive estimates for the Schouten tensor. 
To see Bach equations coupled with $Q$-flat equation also provide estimates of Weyl curvature, one may rewrite Bach 
equation as follows:
\beq\label{Bach-eq-2}
\Delta W_{ijkl} + \nabla_l C_{kji} + \nabla_k C_{lij} + \nabla_iC_{jkl} +\nabla_jC_{ilk} : = K_{ijkl},
\eeq
where $C_{ijk} = A_{ij,k} - A_{ik,j}$ is the Cotton tensor and $K = W*Rm +g^**W*A$ is a quadratic of curvatures (cf. \cite[(2.5)]{CG}). Finally, to get estimates for 
the full Riemann curvature tensor $Rm$, one recalls
that
$$
Rm = W + g^*\owedge A.
$$
The most important analytic tools for elliptic estimates here are Sobolev inequalities. The conditions (4) and (5) in 
Section \ref{introduction and statement of results} (cf. \cite[Theorem 1.1]{CG})
essentially provide the following {\it Sobolev inequality} and  {\it trace Sobolev inequality} respectively:
\beq\label{Sobolev}
(\int_X |u|^\frac {2n}{n-2} d\text{vol}[g^*])^\frac {n-2}n \leq C_s \int_X (|\nabla u|^2) d\text{vol}[g^*]
\eeq
and
\beq\label{trace-Sobolev}
(\oint_{\partial X} |u|^\frac{2(n-1)}{n-2}d\text{vol}[\hat g])^\frac {n-2}{n-1} \leq C_b \int_X (
|\nabla u|^2
)d\text{vol}[g^*]
\eeq
for all $u\in C^1_0(B(p, r_0))$, where $p$ is any point in $\bar X$ and $r_0> 0$ is fixed. Moreover, a global {\it trace Sobolev inequality} holds
\beq\label{trace-Sobolevbis}
(\oint_{\partial X} |u|^\frac{2(n-1)}{n-2}d\text{vol}[\hat g])^\frac {n-2}{n-1} \leq C_b' \int_X (
|\nabla u|^2 + |u|^2
)d\text{vol}[g^*]
\eeq
for all $u\in C^1_0(X)$.

\begin{lemm} \label{epsilonregularity1}(\cite[Theorem 3.4]{CG}) Suppose that $X$ is a smooth 4-manifold with boundary $\partial X$ and 
$g^+$ is a conformally compact Einstein metric on $X$ with the conformal infinity of positive Yamabe type. Let $g^*=\rho^2g^+$ be the 
Fefferman-Graham's compactification associated with the Yamabe metric of the conformal infinity. Assume the Sobolev inequality 
\eqref{Sobolev} holds for the Fefferman-Graham's compactification $g^*$. Then there exists constants $\varepsilon > 0$ and $C_k >0$ 
such that if 
$$
\|Rm\|_{L^2(B(p,r))}\le \varepsilon
$$
for a geodesic ball $B(p, r)\subset X$, then, for each $k = 0, 1, 2, \cdots$, 
\beq\label{Bach-estimates}
 \sup_{B(p, r/2)} |\nabla^{k} Rm| \le  \frac{C_k}{r^{k+2}}\left(\int_{B(p,r)}|Rm|^2 d\text{vol}[g^*]\right)^{\frac12}.
 \eeq  
\end{lemm}

\begin{lemm} \label{epsilonregularity2} (\cite[Theorem 3.1 and Theorem 3.2]{CG}) Suppose that $X$ is a 
smooth 4-manifold with boundary $\partial X$ and 
$g^+$ is a conformally compact Einstein metric on $X$ with the conformal infinity of positive Yamabe type. Let $g^*=\rho^2g^+$ be the 
Fefferman-Graham's compactification associated with the Yamabe metric of the conformal infinity. Assume the Sobolev inequalities \eqref{Sobolev} and \eqref{trace-Sobolev} hold for the Fefferman-Graham's compactification $g^*$. Then there exists constants 
$\varepsilon >0$ and $C _k > 0$ such that if 
$$
\|Rm\|_{L^2(B(p,r))}\le \varepsilon
$$
for a geodesic ball $B(p, r) \subset \bar X$, then, for each $k = 0, 1, 2, \cdots$, 
\beq\label{Bach-estimates-b}
 \sup_{B(p, r/2)} |\nabla^{k} Rm|  \le  \frac{C_k}{r^{k+2}}\left(\int_{B(p, r)}|Rm|^2 d\text{vol}[g^*] 
 + \oint_{B(p, r)\cap \partial X} |S|d\text{vol} [\hat g] + \text{vol}(B(p, r))\right)^{\frac12}.
\eeq  
\end{lemm}

\subsection{Sobolev inequalities}
In the above estimates, Sobolev inequalities are essential.  We need to control the constants $C_s$ and $C_b$ in Sobolev inequalities \eqref{Sobolev} and \eqref{trace-Sobolev} in terms of other geometric quantities in Riemannian geometry. For the convenience of readers, we recall the following 
notions of injectivity radii for a Riemannian manifold $(X, g)$ with boundary $\partial X$. For any interior point $p\in X$, let $i_{\text{int}}(p, g)$
be the supremum of $r$ such that the normal geodesic $\gamma (t)$ from $p$ is minimizing for any $t\in [0, \min\{r, t_\gamma\}]$, where 
$t_\gamma$ is the first intersection of $\gamma$ with the boundary $\partial X$. Then the interior injectivity radius is defined by
$$
i_{\text{int}} (X, g) = \inf \{i_{\text{int}}(p, g): p\in X\}.
$$
For $p\in \partial X$, let $i_{\partial}(p, g)$ be the supremum of $r$ such that the normal geodesic $\gamma$ from $p$ in the inward unit normal direction $\nu_p$ is minimizing for any $t\in [0, r]$. Then the boundary injectivity radius is defined by
$$
i_\partial (X, g) = \inf \{i_{\partial} (p, g): p \in \partial X\}.
$$
The other equivalent definition for the boundary injectivity radius is that $i_\p (X, g)$ is the supremum of the height $h$ of the Fermi coordinates
from the boundary $\p X$ in $X$: 
$$\text{exp}_p (s\nu_p): \p X\times [0, h) \to X
$$
for $p\in \p X$ and $s\in [0, h)$ (cf. \cite{Kodani} \cite[Section3.6
]{Chavel}). 

\begin{lemm} \label{lem4.2}
Let $(X^n, g)$ be a complete Riemannian $n$-manifolds with totally geodesic boundary. Suppose that $|\operatorname{Rm}| \leq k $ 
and that 
$$
i_{\text{int}}(X, g) \ge i_0, \quad i_\p (X, g) \ge i_0, \quad \text{and} \quad i(\p X) \ge i_0.
$$
for a positive constant $i_0$, where $i(\p X)$ is the intrinsic injectivity radius of the boundary. Then the Sobolev inequalities \eqref{Sobolev} and \eqref{trace-Sobolev} (resp. \eqref{trace-Sobolevbis}) hold for uniform constants $C_s$ and $C_b$ (resp. $C_b'$).
\end{lemm}

\begin{proof} We consider the doubling $\tilde X = X\cup_{\p X} X$: the union of two copies of $X$ along the boundary $\p X$ 
where the second $X$ is the reflexion of $X$. It is easy to see that $i(\tilde X, \tilde g) \ge i_0$. Then, \eqref{Sobolev} (local and global) 
simply follows from \cite[Theorem 1]{Cantor} (see also \cite{Aubin} and other related results \cite[Theorem 3.14 and Lemma 3.17]{Hebey}).
\\

For the trace Sobolev inequality \eqref{trace-Sobolev}, one may first use \cite[Theorem A]{Kodani} to find uniform Lipschitz boundary local coordinate system in which the trace Sobolev inequality \eqref{trace-Sobolev} is valid with uniform constant $C_b$ at least for the local 
version. 
\\

To prove that \eqref{trace-Sobolevbis} holds globally, we work with a partition of unity associated with a countable coordinate chart covering 
$\{ B(x_i, \delta/2)\}$, where $(x_i)$ be a sequence of points in $\bar X$, such that 
$$
\bar X=\cup_i B(x_i,\delta/2) \mbox{ and }B(x_i,\delta/4) \cap B(x_j,\delta/4)=\emptyset\mbox{  if  }i\neq j.
$$
Then there exists $N=N(n,k,i_0)$, depending on $n, k, v$, such that  each point of $\tilde X$  has a neighborhood which intersects at 
most $N$ of the balls $B(x_i,\delta)$'s. This comes from Gromov-Bishop volume comparison theorem. Meanwhile, if let $K$ be the 
total number of $B(x_i, \delta/2)$ that intersects with $B(p, r_0)\cap\p X$, then $K$ depends only of $r_0$ and $\delta$.
\\ 

Let $\xi $ be some non-negative cut-off function such that $\xi(t)=1$ on $[0,\delta/2]$ and  $\xi(t)=0$ on $[3\delta/4,+\infty)$, 
and it satisfies $|\xi'|\le C/\delta$ on $[0,+\infty)$. Let $\alpha_i(x)=\xi(d(x,x_i))$ and $\eta_i=\alpha_i/\sum_m \alpha_m$. 
Let $u\in C^1(X)$. We can estimate
$$
\begin{array}{llll}
\ds \left(\oint_{\p X} |u|^{\frac{2(n-1)}{n-2}}d\sigma (g)\right)^{\frac{n-2}{n-1}}
& \le & \ds  (\sum_i  \left(\oint_{\p X} |\eta_i u|^{\frac{2(n-1)}{n-2}}d\sigma (g)\right)^{\frac{n-2}{n-1}})^2\\
&\le &\ds (\sum_i (\int_X |\nabla (\eta_i u)|^2dv (g))^\frac 12)^2 \\
&\le &\ds CKN \int_X( |\nabla u|^2+u^2)dv (g)
\end{array}
$$
Thus the proof is complete.
\end{proof}

\begin{rema} In the recent paper \cite{CLW}, the authors have established the remarkable inequality that 
$$
Y_b(X, \p X, [g^*])^2 \ge 6 Y(\p X, [\hat g])
$$
for any conformally compact Einstein manifold $(X, g^+)$ with its conformal infinity of positive Yamabe type.
We remark that as a direct consequence one can drop the assumption (5) in the statements of the main theorems in \cite{CG}  . In other words, the global trace-Sobolev inequality
(5) (therefore \eqref{trace-Sobolevbis}) is always available for any conformally compact Einstein manifold $(X, g^+)$ with its 
conformal infinity of positive Yamabe type. Thus the effort in the current paper is to drop the assumption (4), although the same procedure also works to drop assumption (5) in
the statements of both Theorem \ref{maintheorem1} and  Theorem \ref{maintheorem2}.
\end{rema}

\subsection{Cheeger-Gromov convergences for manifolds with boundary} Our approach to establish the compactness of 
conformally compact Einstein 
4-manifolds is to prove by contradiction. We will analyze and eliminate the causes of possible non-compactness by the method of blow-up.
This method has been essential and powerful in many compactness problems in geometric analysis, particularly in Riemannian geometry.
The fundamental tool in the context of Riemannian geometry is the so-called Cheeger-Gromov convergences of Riemannian
manifolds developed from Gromov-Hausdorff convergences (see, for example, \cite{CGT, Anderson0}, for 
Cheeger-Gromov convergences of Riemannian
manifolds without boundary). In this subsection, for later uses in our paper, we will present the Cheeger-Gromov convergences 
for manifolds with boundary. Good references in the subject are  for examples in \cite{Perales, Kodani, Knox, Wong, AKKLT}.
\\

Let us first recall the definition of harmonic radius for a Riemannian manifold with boundary (cf. \cite{Perales}). 
Assume $(X, \ g)$ is a complete Riemnnian $4$-manifold with the boundary $\p X$. A local coordinates 
$$
(x_0,x_1,x_2,x_3): B(p, r) \to \Omega \subset \mathbb{R}^4
$$
is said to be harmonic if, 
\begin{itemize}
\item $\triangle x_i=0$ for all $0\le i\le 3$ in $B(p, r) \subset X$, when $p\in X$ is in the interior;
\item $\Delta x_i = 0 $ for all $0\le i \le 3$ in $B(p, r)\cap X$ and,  on the boundary $B(p, r)\cap \p X$,  $(x_1, x_2, x_3)$ is a harmonic
coordinate in $\p X$ at $p$ while $x_0 = 0$, when $p\in \p X$ is on the boundary.
\end{itemize}
For $\alpha\in (0,1)$ and $ M \in (1,2)$, we define the harmonic radius $r^{1,\alpha}(M)$ to be the biggest number $r$ satisfying the following properties:
\begin{itemize}
\item If $\text{dist}(p, \p X) > r$, there is a harmonic coordinate chart on $B(p, r)$ such that
\beq
\label{relation1}
M^{-2}\delta_{jk} \le g_{jk}(x) \le M^{2}\delta_{jk}
\eeq
and
\beq
\label{relation2}
r^{1+\alpha}\sup|x-y|^{-\alpha}|\p g_{jk}(x)-\p g_{jk}(y)|\le M-1
\eeq
in $\overline{B(p, \frac r2)}$.
\item If $p\in\p X$, there is a boundary harmonic coordinate chart on $B(p, 4r)$ such that \eqref{relation1} and \eqref{relation2}
hold in $\overline{B(p, 2r)}$. 
\end{itemize}

\vskip 0.2in
The following is the extension of the $C^{1, \alpha}$ convergence theorem of Anderson \cite{CGT, Anderson0} to manifolds with boundary
(cf. \cite{Kodani, AKKLT}).

\begin{lemm}(\cite[Theorem 3.1]{AKKLT})  \label{AKKLT}
Suppose that ${\mathcal M}(R_0,i_0, h_0, d_0)$ is the set of all compact Riemannian manifolds $(X, g)$ with boundary such that 
$$
\aligned
|Ric_X|\le R_0, & \quad |Ric_{\p X}|\le R_0\\
i_{\text{int}}(X) \ge i_0, & \quad  i_\p (X) \ge 2i_0, \quad i(\p X) \ge i_0,\\
\text{Diam}(X) \le d_0, & \quad \| H \|_{Lip(\p X)} \le h_0,
\endaligned 
$$
where $Ric_{\p X}$ is the Ricci curvature of the boundary, $i(\p X)$ is the injectivity radius of the boundary, and $H$ is the mean 
curvature of the boundary.  Then ${\mathcal M}(R_0,i_0, h_0, d_0)$  is pre-compact in the $C^{1,\alpha}$ Cheeger-Gromov topology for 
any $\alpha\in (0,1)$.
\end{lemm}

\begin{rema}\label{Rmk:AKKLT} We remark
\begin{itemize}
\item First, in \cite{AKKLT}, it is showed that the harmonic radius $r^{1, \alpha} (M)$ is uniformly bounded from 
below in ${\mathcal M}(R_0,i_0, h_0, d_0)$ (cf. \cite[Theorem 3.2.1]{AKKLT}). 
\item Second,  it is easy to see that, after having 
harmonic coordinate charts with the uniform size, one has the pre-compactness in $C^{k+2, \alpha}$ Cheeger-Gromov topology 
if the Ricci 
curvatures are bounded in $C^{k, \alpha}$ norm and the boundaries are all totally geodesic (even for $k=0$), which is the convergence theorem
that is useful to us later (see  \cite{CG}). 
\item Third, one may have the pre-compactness in the Cheeger-Gromov topology with base points if dropping
the assumption on the diameter $\text{Diam}(X)$.
\end{itemize}
\end{rema}


\section{Injectivity radii: blow-up before blow-up}\label{Sect:injrad}

Our main results in this section concern the injectivity radius estimates for manifolds with boundary. For our purpose we may always assume
that the geometry of the boundary is compact in Cheeger-Gromov sense. The following is an easy consequence from 
\cite[Theorem 3.1]{AKKLT}, which was mentioned as Lemma \ref{AKKLT} in the previous Section \ref{Sect:prelim}.

\begin{lemm} 
\label{bdy-injrad}
Suppose that $(X^4, g^+)$ is a conformally compact Einstein 4-manifold with the conformal infinity of Yamabe constant $Y(\p
X, [\hat g]) \ge Y_0 >0$. And suppose that $(X^4, g^*)$ is the Fefferman-Graham's compactification associated with the Yamabe 
metric $\hat g$ on the boundary such that the intrinsic injectivity radius $i(\p X, \hat g) \ge i_o>0$,  and that $ i_\p (X, g^*)\le  i_{\text{int}} (X, g^*)$. Then there is a constant
$C_\p = C(n) > 0$, depending of $i_0$, such that 
\begin{equation}\label{Eq: bdy-injrad}
\max_X |Rm| (i_\p (X, g^*))^2  + i_\p (X, g^*) \ge C_\p
\end{equation}
where $Rm$ is Riemann curvature of $g^*$.
\end{lemm}

\begin{proof} We show by contradiction. Suppose otherwise there are a sequence of Fefferman-Graham's compactified Riemannian manifolds $(X_i, g_i^*)$
such that 
$$
\max_{X_i} |Rm_{g_i^*}| (i_\p (X_i, g_i^*))^2  + i_\p (X_i, g_i^*)  \to 0
$$
and $i(\p X_i) \ge i_0$ for some fixed positive number $i_0$. We then rescale the metrics as follows:
$$
\bar g_i= K^{-2}_i g_i^*
$$
where $K_i = i_\p (p_i) = i_\p (X_i, g_i^*)$ for some $p_i\in \p X_i$. Here we use the fact that the boundary injectivity radius $i_\p(\cdot) $ 
is a continuous function on the boundary since the limit of minimizing geodesics is still minimizing geodesic. Recall the boundary of $g_i^*$ is totally geodesic. On the other hand, because the 
curvature 
$$
\max_{X_i} |Rm_{\bar g_i}| =  \max_{X_i} |Rm_{g_i^*}|  
(i_\partial(X_i, g_i^*))^2  \to 0,
$$
by \cite[Lemma 6.3]{Kodani}, there is a normal geodesic $\gamma$ of length $2$ such that $\gamma$ is orthogonal to 
boundary $\p X_i$ at $\gamma(0)= p_i$ and $\gamma(2)= p_i'$.
\\ 

In light of Lemma \ref{AKKLT}, we may extract a subsequence  (we will always use the same index for subsequences for 
convenience in this paper) $(X_i, \bar g_i, p_i)$ that converges to $(X^n_\infty, g_\infty, p_\infty)$  in $C^{1, \alpha}$ Cheeger-Gromov 
topology. From the assumptions, it is
easily seen that $(X^n_\infty, g_\infty, p_\infty)$ is a complete flat metric manifold with the totally geodesic complete flat boundary 
(it is smooth in harmonic coordinates as demonstrated in \cite{Anderson0, AKKLT}). 
First,  each connected component of the boundary $(\p X_\infty, \hat g_\infty)$ is the Euclidean space $\mathbb{R}^{n-1}$ because of $i(\p X_i) \ge i_0$. Second,  by the Cheeger and Gromoll's splitting theorem, the complete metric $g_\infty$ is a product metric on 
$\mathbb{R}^{n-1} \times (0, \infty)$ or $\mathbb{R}^{n-1} \times (0, 2)$. We claim the later case does not appear, that is, $X_\infty=\mathbb{R}^{n-1} \times (0, \infty)$.  Let $\bar\rho_i$ be the Fefferman-Graham's defining function related to the metric $\bar g_i$. As in the proof of  \cite[Lemma 4.3]{CG}, we know, up to a subsequence, $\bar \rho_i\to \rho_\infty$ and the hessian of $\rho_\infty$ vanishes since $g_\infty$ is flat. Moreover, $\rho_\infty(x)\ge C\text{dist}_{g_\infty}(x, \p X_\infty)$ for some positive constant $C>0$. In view of the boundary condition, $\rho_\infty$ vanishes on $\mathbb{R}^{n-1} \times \{0\}$. Hence, $\rho_\infty(x_1,\cdots,x_n)=\alpha x_n$ with some positive constant $\alpha>0$ since $\nabla^2 \rho_\infty=0$. When $\p X_\infty$ has more than 2 connected components, this contradicts the fact $\rho_\infty$ vanishes on $\p X_\infty$. Therefore, we prove the claim.\\
Now, on the other hand, there is a geodesic of length 2 in $(X_\infty, g_\infty)$ which are orthogonal to the boundary $\p X_\infty$. 
This is a contradiction. Therefore, we have established the Lemma.
\end{proof}

Next we would like to get the lower bound estimates for the interior injectivity radius $i_{\text{int}}$ of a compact Riemannian
manifold with boundary. The real reason for having no interior collapsing follows from the following recent work in \cite{LQS}.

\begin{lemm}\label{Lem:LQS} (Li-Qing-Shi \cite[Theorem 1.3]{LQS}) 
Suppose that $(X^4, g^+)$ is a conformally compact Einstein manifold with the conformal infinity of
Yamabe constant $Y(\p X, [\hat g]) > 0$. Then, for any $p\in X^4$,
\begin{equation}\label{Eq:LQS}
\frac {\text{vol}_{g^+}(B(p, r))}{\text{vol}_{g_{\mathbb{H}^4}}(B(r))} \ge (\frac {Y(\p X, [\hat g])}
{Y(\mathbb{S}^{3}, [g_{\mathbb{S}}])})^\frac 32 
\end{equation}
\end{lemm}

As a consequence, we have

\begin{lemm} 
\label{int-injrad}
Suppose that $(X^4, g^+)$ is a conformally compact Einstein 4-manifold with the conformal infinity of Yamabe constant $Y(\p
X, [\hat g]) \ge Y_0 >0$.  And suppose  that $(X^4, g^*)$ is the Fefferman-Graham's compactification associated with the Yamabe 
metric $\hat g$ on the boundary such that the intrinsic injectivity radius $i(\p X, \hat g) \ge i_o>0$, and that $ i_\p (X, g^*)\ge  i_{\text{int}} (X, g^*)$. Then there is a constant
$C_{\text{int}} > 0$, depending of $Y_0$ and $i_0$, such that 
\begin{equation}\label{Eq: int-injrad}
\max_X |Rm|(i_{\text{int}} (X, g^*))^2  + i_{\text{int}} (X, g^*) \ge C_{\text{int}}
\end{equation}
where $Rm$ is the Riemann curvature of $g^*$.
\end{lemm}
\begin{proof} Again, 
we will prove this lemma by contradiction. Assume otherwise, then there is a sequence of conformally compact Einstein 4-manifolds 
$(X_i^4, g_i^+)$ with the conformal infinity of Yamabe constants $Y(\p X_i, [\hat g_i]) \ge Y_0>0$, whose 
Fefferman-Graham's compactifications $(X_i^4, g_i^*)$ associated with the Yamabe metrics $\hat g_i$ 
on the boundary are compact 4-manifolds with totally geodesic boundary and satisfy  
$$
\max_{X_i} |Rm_{g_i}|(i_{\text{int}} (X_i, g_i^*))^2  + i_{\text{int}} (X_i, g_i^*) \to 0
$$
and 
$$
i(\p X_i, \hat g_i) \ge i_0.
$$

Let us consider the rescaling
$$
\bar g_i = K_i^{-2}g_i^*
$$
for $K_i = i_{\text{int}}(X_i, g_i^*)$. Using \cite[Lemma 6.4]{Kodani} to the almost flat metrics $\bar g_i$,  
one sees that $1 = i(X_i, \bar g_i) = i_{\text{int}} (p_i, \bar g_i)$ for some $p_i \in X_i$  in the interior. 
Now, if $K^{-1}_i \text{dist}_{g_i^*}(p_i, \p X_i) < \infty$, we are in the same situation as in the 
proof of Lemma \ref{bdy-injrad} and derive the contradiction by \cite[Lemma 6.4]{Kodani} (there would be a closed geodesic of 
length 2 in the Euclidean half space).  
\\

Therefore we may assume that $K^{-1}_i \text{dist}_{g_i^*}(p_i,\p X_i) = \text{dist}_{\bar g_i}(p_i, \p X_i) \to\infty.$
Thus the limit space $(X^n_\infty, g_\infty, p_\infty)$ is a complete flat manifold with no boundary, but, with a simple 
closed geodesic of length $2$. We claim that $(X^4_\infty, g_\infty, p_\infty)$ is of Euclidean volume growth in dimensions 4. 
This would be a contradiction, since such flat manifold would be a product of a circle and a flat manifold of dimension 3, 
which would not be able to support the Euclidean volume growth in dimensions 4. 
\\

To finish the proof of the claim that $(X^4_\infty, g_\infty, p_\infty)$ is of Euclidean volume growth in dimensions 4, that is, 
\begin{equation}\label{Claim-vol}
\text{vol}_{g_\infty}(B^{g_\infty}(p_\infty, r)) \ge c_v r^4
\end{equation}
for some fixed $c_v$ and any $r>0$. First let us prove the following claim.
\\

\noindent{\bf Claim:} $\bar\rho_i (p_i)\to \infty$, where $\bar\rho_i = K_i^{-1}\rho_i$ and $K_i = i_{\text{int}}(X_i, g_i^*)$. 

\begin{proof} Assume otherwise that there is a constant $\bar\rho_0>0$ such that $\bar \rho_i (p_i) \le \bar\rho_0$ for all $i$. Due to \eqref{relation3}
at the beginning of the Section \ref{Subsect:rho}, we have
$$
1 - |\nabla \bar\rho_i|^2_{\bar g_i} = \frac 16 R_{\bar g_i} \bar\rho_i^2 \ge 0,
$$
where the covariant derivatives is calculated with respect to the background metric $\bar g_i$. Let us denote
$$
\epsilon_i = \max\{|Rm_{\bar g_i}|, (\text{dist}_{\bar g_i}(p_i, \p X_i))^{-1}\}\to 0.
$$
Then we obtain
$$
\bar \rho_i (x) \le \frac 12 \epsilon^{-\frac 12}_i + \bar\rho_i(p_i)
$$
for all $x\in B^{\bar g_i}(p_i, \frac 12\epsilon^{-1}_i)\subset X_i$ since $|\nabla \bar\rho_i|_{\bar g_i} \le 1$. This in turn implies
$$
1 \ge |\nabla \bar\rho_i|_{\bar g_i} \ge \frac 12 
$$
for all $x\in B^{\bar g_i}(p_i, \frac 12\epsilon^{-1}_i)\subset X_i$, at least for $i$ sufficiently large. Therefore, along the integral curve 
$\gamma(t)$ of the gradient $\nabla_{\bar g_i}\bar\rho_i$ from $p_i$, we may derive
$$
\bar\rho_i (\gamma(t)) \le \bar\rho_0 - \frac t2 < 0
$$
when $t > 2\bar\rho_0$, which is a contradiction since $\gamma(t) \in X$ for any $t\in (2\bar\rho_0, \frac 12\epsilon_i^{-1})$. So the proof of 
this claim is complete.
\end{proof}

Now let 
$$
s_i = \min\{\text{dist}_{\bar g_i}(p_i, \p X_i), \bar\rho_i(p_i)\}\to \infty.
$$
We find, for each $x\in B^{\bar g_i}(p_i, \frac {s_i}2)$,
$$
|\bar\rho_i(x) - \bar\rho_i(p_i)| \le \frac 12\bar\rho_i(p_i),
$$
which implies,
$$
\frac 12 \le \frac {\bar\rho_i(x)}{\bar\rho_i(p_i)} \le \frac 32.
$$
Notice that $\bar g_i = \bar\rho_i^2 g^+_i$. Now, applying Lemma \ref{Lem:LQS} (cf. \cite[Theorem 1.3]{LQS}), we deduce , 
for $r < \frac {s_i}2$, 
$$
\begin{array}{lll}
\ds \text{vol}_{\bar g_i}(B^{\bar g_i}(p_i, r))&\ge&\ds \text{vol}_{\bar g_i}(B^{g_i^+}(p_i, \frac 23 (\bar\rho_i(p_i))^{-1}r))\\
&\ge & \ds \frac 1{(2(\bar\rho_i(p_i))^{-1})^4} \text{vol}_{g_i^+}(B^{g_i^+}(p_i,  \frac 23 (\bar\rho_i(p_i))^{-1}r))\\
&\ge & \ds  C \frac 1{(2(\bar\rho_i(p_i))^{-1})^4} \text{vol}_{g_{\mathbb{H}^4}}(B^{g_{\mathbb{H}^4}}(\frac 23 (\bar\rho_i(p_i))^{-1}r))\\
&\ge & \ds c_v r^4
 \end{array}
$$
for a fixed constant $c_v$ that is independent of $i$. 
Passing to the limit as $i\to\infty$, we get the desired inequality \eqref{Claim-vol} on the limit space $(X^4_\infty, g_\infty, p_\infty)$.
So the proof is complete.
\end{proof}


\section{On compactness of Fefferman-Graham's compactifications}\label{Sect:compactness}

Based on the preparation in the previous sections we are ready to establish the compactness of Fefferman-Graham's compactifications
on conformally compact Einstein 4-manifolds which were stated in the introduction.  The approach follows closely from the corresponding results in \cite{CG}. 
The difference from \cite{CG} is that, in 
light of the injectivity radius estimates in the previous section, Sobolev inequality and trace Sobolev inequality are all available for the
rescaled metrics with bounded curvature, while Sobolev inequality and trace Sobolev inequality are parts of the assumptions in the
main compactness theorem in \cite{CG}. Readers are referred to \cite{CG} for more details. To begin the proof, we will first establish some bounded 
 curvature estimates.

\begin{lemm} \label{Lem:curv-estimate} 
Suppose that $\{(X_i^4, g^+_i)\}$ is a sequence of conformally compact Einstein 4-manifolds satisfying
the assumptions (1), (2), and (3) in Theorem \ref{maintheorem1}. 
Then there is a positive constant $K_0$ such that, for the Fefferman-Graham's compactifications $\{(X_i^4, g_i^*)\}$ 
associated with the Yamabe metric $\hat g_i$ of the conformal infinity $(\p X_i, [\hat g_i])$ 
\begin{equation}\label{Eq:curvature-bound}
\max_{X_i} |Rm_{g_i^*}| \le K_0
\end{equation}
for all $i$.
\end{lemm}

\begin{proof} Suppose otherwise that there is a subsequence $\{(X_i^4, g_i^+)\}$ satisfying 
$$
K_i  = \max_{X_i} |Rm_{g_i^*}|\to \infty.
$$
Let 
$$
K_i = K_i(p_i) =  |Rm_{g_i^*}| (p_i)
$$
for some $p_i\in \overline{X_i}$. Then we consider the rescaling
$$
(X^4_i, \bar g_i = K_ig_i^*, p_i).
$$
We claim for the metrics $\bar g_i $ the derivatives of curvature $\|\nabla Rm_{\bar g_i}\|$ is uniformly bounded.  \\
For this purpose, we recall  the  curvature tensor satisfies some elliptic PDE with the Dirichlet boundary conditions \cite[section 4]{CG}
\beq
\label{ellipticsystem}
\left\{
\begin{array}{llll}
\triangle R=R^2-3|Ric|^2&\mbox{ in }X\\
R=3\hat{R}&\mbox{ on }\p X\\
\triangle A-\frac16 \nabla^2 R=Rm*A&\mbox{ in }X\\
 A_{\alpha\beta}=\hat{A}_{\alpha\beta}, A_{\alpha n}=0, A_{\alpha n}, A_{n n}=\frac{\hat{R}}{4}&\mbox{ on }\p X\\
\triangle W=Rm*W+g*W*A+L(Ric)&\mbox{ in }X\\ 
W=0 &\mbox{ on }\p X,
\end{array}
\right.
\eeq
where $L$ is a linear differential operator of second order.  In view of \cite[Theorem 8.33]{GT} we have the boundedness of the curvature tensor in $C^{1,\alpha}$ topology since  the curvature tensor on the boundary $\hat{Rm}$ is bounded in $C^{1,\alpha}$ topology.\\

\noindent{\bf No boundary Blow-up:} Let us first consider the cases where 
$$
\text{dist}_{\bar g_i} (p_i, \p X_i) < \infty.
$$
For the pointed manifolds $(X_i, \bar g_i, p_i)$ with boundary, in the light of all the preparations in the previous sections, 
particularly Lemma \ref{AKKLT} and Remark \ref{Rmk:AKKLT}, we have 
Cheeger-Gromov convergence
$$
(X^4_i, \bar g_i, p_i) \to (X^4_\infty, g_\infty, p_\infty)
$$
in 
$C^{k, \alpha}$  Cheeger-Gromov topology (for a subsequence if necessary), where the limit space is a 
complete Bach-flat and $Q$-flat manifold with a totally geodesic boundary $\p X_\infty$; the boundary 
$(\p X_\infty, \hat g_\infty)$ is simply the Euclidean space $\mathbb{R}^{n-1}$ because $i(\p X_i)\ge i_0>0$; and 
$$
|Rm_{g_\infty}|(p_\infty)= 1.
$$ 
To derive the a priori estimates for Cheeger-Gromov convergence, one applies the $\epsilon$-estimates in 
Lemma \ref{epsilonregularity1} and Lemma \ref{epsilonregularity2}, where Sobolev inequality \eqref{Sobolev} and trace Sobolev 
inequality \eqref{trace-Sobolev} are established in Lemma \ref{lem4.2}. The injectivity radii estimates that are needed for $\bar g_i$ 
to satisfy Sobolev and trace Sobolev are given in Lemma \ref{bdy-injrad} and Lemma \ref{int-injrad}.
\\

Now, clearly, to finish the proof is to show that the limit space $(X^4_\infty, g_\infty, p_\infty)$ is the Euclidean half space. 
For the convenience of readers, we very briefly sketch the proof from \cite{CG}. One first needs to show that 
$\bar\rho_i \to \rho_\infty$ where $\rho_\infty$ satisfies
\begin{itemize}
\item $g^+_\infty = \rho_\infty^{-2}g_\infty$ is a (partially) conformally compact Einstein metric on $X^4_\infty$ whose conformal
infinity is the Euclidean space $\mathbb{R}^{n-1}$;
\item $- \Delta_{g^+_\infty} \log\rho_\infty = 3$.
\end{itemize}
Then, by Condition (2) in Theorem \ref{maintheorem1}, one shows that $g^+_\infty$ is locally hyperbolic space metric nearby 
the infinity $\p X^4_\infty = \mathbb{R}^{n-1}$  based on the 
unique continuation therem in \cite{Biquard, BiquardHerzlich}. Finally one concludes that $\rho_\infty = x_0$, since 
$(X^4_\infty, g^+_\infty)$ is hyperbolic space in half space model $(\mathbb{R}^4_+, \frac {|dx|^2}{x_0^2})$ for 
$\mathbb{R}^4_+ = \{x = (x_0, x_1, x_2, x_3)\in \mathbb{R}^4: x_0  >0\}$, which implies that $(X^4_\infty, g_\infty)$ 
is the Euclidean half space (please see the details in the proof of \cite[Proposition 4.8]{CG}). 
\\

\noindent{\bf No interior blow-up:} Next we consider the rest cases when
$$
\text{dist}_{\bar g_i}(p_i, \p X_i) \to \infty
$$
(at least for some subsequence). Notice that, 
$$
K_i = \max_{X_i} |Rm_{g_i^*}| =|Rm_{g_i^*}| (p_i)
$$
for some $p_i\in X$ in the interior. Proceeding as the above boundary cases, one has the Cheeger-Gromov convergence
$$
(X^4_i, \bar g_i, p_i) \to (X^4_\infty, g_\infty, p_\infty)
$$
in 
 $C^{k, \alpha}$  Cheeger-Gromov topology. The proof in these cases follows from \cite{CG}.  We again very briefly sketch the proof that is 
more or less from \cite{CG}. One first derives from \eqref{relation3} that
$$
R_{\bar g_i} = \bar\rho_i^{-2} (1 - |d\bar\rho_i|^2_{\bar g_i})
$$
and shows that
\begin{itemize}
\item $\bar\rho_i (x) \ge C \text{dist}_{\bar g_i} (x, \p X_i)$. (cf. Step 2 in the proof of \cite[Lemma 4.9]{CG}).
\end{itemize}
Then, consequently, 
\begin{itemize}
\item $R_\infty =0$, and 
\item $g_\infty$ is Ricci-flat from being $Q$-flat and scalar flat in the light of the $Q$-curvature equation \eqref{Q-eq}. (cf. Step 3 of
the proof of \cite[Lemma 4.9]{CG}).
\end{itemize}
Thus, $(X_\infty, g_\infty)$ is a complete Ricci-flat 4-manifold with no boundary. 
At this point, as argued in \cite{CG}, first, due to the recent work in \cite{CheegerNaber}, one concludes that 
$(X_\infty, g_\infty)$ is a complete  ALE Ricci flat 4-manifold. By the assumptions, the doubling of $X$ is a homological sphere.  By a topological result due to Crisp-Hillman (\cite{CH} Theorem 2.2), $(X_\infty, g_\infty)$ at the infinity is asymptotic to $\mathbb{S}^3/\Gamma$ with $\Gamma=\{1\}$ or $\Gamma=Q_8$(quaternion
group) or $ \Gamma$ the perfect group (that is, $\mathbb{S}^3/\Gamma$ is a homology 3-sphere). 
 By the Chern-Gauss-Bonnet formula and the signature formula, we obtain the desired contradiction. 
 For more details see \cite{CG} section 4.3.
\end{proof}

We now begin the proof of Theorem \ref{maintheorem1}.
\vskip .1in

With the curvature bound \eqref{Eq:curvature-bound}, the injectivity radius estimates in Lemma \ref{bdy-injrad} and Lemma \ref{int-injrad}, 
the $\epsilon$-regularities Lemma \ref{epsilonregularity1} and Lemma \ref{epsilonregularity2}, one last piece that is needed to apply the 
Cheeger-Gromov convergences for manifolds with boundary in Lemma \ref{AKKLT} and Remark \ref{Rmk:AKKLT} to finish the proof
of Theorem \ref{maintheorem1} is the following diameter bound.

\begin{lemm} 
\label{diameter}
Under the assumptions in Theorem \ref{maintheorem1}, the diameters of the Fefferman-Graham's compactifications
$g_i^*$ are uniformly bounded. 
\end{lemm}

\begin{proof} The proof of this lemma under the  assumption that the first Yamabe invariant is uniformly  bounded from below is obtained in \cite[Section 5: {\it The proof of Theorem 1.1}]{CG}. However,  we do not know if one has  the suitable Euclidean Sobolev type inequality in actual setting. This makes the problem is more delicate. Here we give a different approach to overcome the difficulty.
\\

We have already proved the family of metrics $ g_i^*$ has the bounded curvature in $C^{1}$ so that the arguments given in  \cite[Section 4.4: {\it The proof of Theorem 4.4}]{CG} yields the bound in $C^{k+1}$ norm.  In view of Lemmas \ref{bdy-injrad} and \ref{int-injrad}, the boundary radius and the interior one are uniformly bounded from below.  Therefore, for all $i$, for all $x\in \bar X$, we have $vol(B^{g_i^*}(x,1))\ge C>0$ for some constant $C>0$ independent of $i,x$, that is, there is non-collapse.  We prove the diameter is uniformly bounded from above by contradiction. Suppose that the diameter $\diam(g_i^*)$ tends to the infinity. By Cheeger-Gromov-Hausdorff compactness theory, up to diffeomorphisms fixing the boundary, $(X_i,g_i^*)$ converges to some complete non-compact manifold $(X_\infty, g_\infty)$ with the boundary. We divide the proof in 5 steps.\\

{\it Step 1.} There exists some $C>0$ such that $\rho_i\ge C$ provided $d_{ g_i^*}(x,\partial X)\ge 1$ and $d_{ g_i^*}(x,\partial X)\le C\rho_i(x)$ provided $0\le d_{ g_i^*}(x,\partial X)\le 1$.Thus the limit metric is conformal to an asymptotic hyperbolic Einstein manifold.  Moreover, there exists some constant $C_1>0$ independent of $i$ such that $\ds\int |Rm_{g_i^*}|^2\le C_1$. 
\\

The first part of the claim can be proved in the same way as in \cite[Section 4: {\it the proof of Lemma 4.4}]{CG}. The second part is proved  in \cite[Section 5: {\it the step 2 of the proof of Theorem 1.1}]{CG}.
Without loss of generality,  assume the boundary injectivity radius is bigger than 1.
 \\

{\it Step 2.} There exists some constant $C_2>0$ independent of $i$ such that 
\beq
\label{eqlem4.2}
\ds\int_{\{x, d_{ g_i^*}(x,\p X)\ge 1\}} \rho^{-2}_i (x)\le C_2.
\eeq
Thanks of (\ref{relation3}) and (\ref{relation4}),  we infer
$$
-\triangle \log \rho_i=\frac {R_i}2+\frac{|\nabla \rho_i|^2}{\rho_i^2}=3\frac{(1-|\nabla \rho_i|^2)}{\rho_i^2}+\frac{|\nabla \rho_i|^2}{\rho_i^2}
$$
Integrating on the set ${\{x, d_{ g_i^*}(x,\p X)\ge 1\}}$, we obtain
$$
\int_{\{x, d_{ g_i^*}(x,\p X)\ge 1\}} 3\frac{(1-|\nabla \rho_i|^2)}{\rho_i^2}+\frac{|\nabla \rho_i|^2}{\rho_i^2} =\left| \oint_{\{x, d_{ g_i^*}(x,\p X)= 1\}}\frac 1{\rho_i}  \langle \nabla \rho_i ,\nu\rangle\right |
$$
where $\nu$ is the outside normal  vector on the boundary ${\{x, d_{ g_i^*}(x,\p X)= 1\}}$. By Step 1, we know ${\rho_i} $ is uniformly bounded from below on the set ${\{x, d_{ g_i^*}(x,\p X)= 1\}}$. Together the facts the curvature of $g_i^*$ is bounded and the boundary $(\p X_i,\hat{g_i})$ is compact, we infer for some positive constant $C>0$
$$
\int_{\{x, d_{ g_i^*}(x,\p X)\ge 1\}} 3\frac{(1-|\nabla \rho_i|^2)}{\rho_i^2} \le C,\mbox{ and }  \int_{\{x, d_{ g_i^*}(x,\p X)\ge 1\}} \frac{|\nabla \rho_i|^2}{\rho_i^2} \le C
$$
since $|\nabla \rho_i|\le 1$. Combining these estimates, the desired claim yields.\\

{\it Step 3.}  We have 
$$
\lim_{x\to\infty} \rho_\infty(x) =+\infty
$$
Letting $i\to\infty$ in (\ref{eqlem4.2}), we get
\beq
\label{eqbislem4.2}
\int_{\{x, d_{ g_\infty}(x,\p X)\ge 1\}} \rho^{-2}_\infty(x)\le \lim_i\int_{\{x, d_{ g_i^*}(x,\p X)\ge 1\}} \rho^{-2}_i (x)\le C_2.
\eeq
For all $\varepsilon >0$, there exists $A>0$ such that
$$
\int_{\{x, d_{ g_\infty}(x,\p X)\ge A\}} \rho^{-2}_\infty(x)\le \varepsilon 
$$
Therefore, for any $y$ with $d_{ g_\infty}(y,\p X)\ge A+1$, we can estimate
$$
\int_{B^{g_\infty}(y,1)} \rho^{-2}_\infty(x)\le \int_{\{x, d_{ g_\infty}(x,\p X)\ge A\}} \rho^{-2}_\infty(x)\le  \varepsilon 
$$
so that 
$$
(\sup_{B^{g_\infty}(y,1)} \rho_\infty)^{-2} Vol (B^{g_\infty}(y,1))\le  \varepsilon 
$$
that is,
$$
\sup_{B^{g_\infty}(y,1)} \rho_\infty\ge C\varepsilon^{-1/2}
$$
Together with Lemma \ref{lem4.1}, we deduce
$$
\inf_{B^{g_\infty}(y,1)} \rho_\infty\ge \sup_{B^{g_\infty}(y,1)} \rho_\infty-1\ge C\varepsilon^{-1/2}-1
$$
Finally, we prove Step 3.\\

{\it Step 4.}  We claim that there exists some $c_v>0$ such that for any $p\in X_\infty$ and for any $r<\frac12
\rho_\infty(p)$
\beq
\label{eqlembisbis4.2}
Vol(\text{dist}(B^{g_\infty}(p,r)))\ge c_v r^4
\eeq

Let $p_i\in X_i $ such that $p_i\to p$.  First we remark that $\text{dist}_{g_i^*}(p_i, \p X_i)\ge  \rho_i(p_i)$ because of  Lemma \ref{lem4.1}. As in the proof of the end of Section 3, we have 
$$
Vol(\text{dist}(B^{g_i^*}(p_i,r)))\ge c_v r^4,
$$
where $c_v$ is some positive constant independent of $i$. Letting $i\to\infty$, the claim is proved.\\

{\it Step 5.} A contradiction.\\

On choose $p\in X_\infty$ such that $\rho_\infty (p)$ is sufficiently large. We fix $r=(\rho_\infty (p))^{2/3}$.  Using the results in Steps 2 and 4, we get
$\rho_\infty(p)$
\beq
\label{eqlembisbis24.2}
(\sup_{B^{g_\infty}(p,r)} \rho_\infty)^{-2} Vol (B^{g_\infty}(p,r))\le  \int_{B^{g_\infty}(p,r)} \rho^{-2}_\infty(x)\le C_2
\eeq
so that for some  positive contsant  $C>0$ there holds
\beq
\label{eqlembisbis34.2}
\sup_{B^{g_\infty}(p,r)} \rho_\infty\ge  Cr^2= C(\rho_\infty (p))^{4/3}
\eeq
On the other hand, it follows from Lemma \ref{lem4.1}, we deduce
\beq
\label{eqlembisbis34.2bis}
\inf_{B^{g_\infty}(p,r)} \rho_\infty\ge \sup_{B^{g_\infty}(p,r)} \rho_\infty-r
\eeq
so that
$$
\rho_\infty (p) +(\rho_\infty (p))^{2/3}=\rho_\infty (p)+r\ge C(\rho_\infty (p))^{4/3}.
$$
This yields that $\rho_\infty (p)$ is bounded. This contradicts the claim in Step 3. \\
Thus we have finished the proof of Lemma \ref{diameter}, hence the proof of  Theorem \ref{maintheorem1}.
\end{proof}

The proof of Theorem  \ref{maintheorem2} is quite similar to the corresponding theorem in \cite{CG} albeit the removing 
of conditions (4) and (5).   We leave the details to the readers.  The proof of Theorem  \ref{maintheorem3}
is different as we have no informations on the $S$-tensor or $T$ curvature in the statement  of the theorem. 
Here we will present the proof. 
\\

\begin{proof}[Proof of Theorem \ref{maintheorem3}] 
Again we will prove by contradiction.   Let $\{g_i^+\}$ be a set of conformally compact Einstein metrics on $X^4$ and $\{g_i^*\}$ corresponding 
Fefferman-Graham's compactifications associated with the Yamabe metrics $\hat g_i\in \mathcal{C}$, where $\mathcal{C}$ is compact
in $C^{k,\alpha}$ Cheeger-Gromov topology as given in \eqref{Con:1} in Theorem \ref{maintheorem1}. Assume that 
$$
\text{($2^\star$)} \quad \text{either } Y(\partial X, [\hat g_i]) \to Y(\mathbb{S}^3, [g_{\mathbb{S}}]) \text{ or } \int_X (|W|^2dvol)[g^+_i]\to 0,
$$ 
but $(X, g_i^*)$ does not converges in $C^{k,\alpha}$ Cheeger-Gromov topology. If the interior blow-up were to happen, 
then, it is easily seen that it would be a contradiction with the fact that any possibly limit space is flat due to ($2^\star$), in light of 
the rigidity in Gromov-Bishop's volume comparison principle or simply the limit metric is both Ricci flat and locally conformally flat. If the 
boundary blow-up were to happen, then it is again easily seen that it would be a contradiction with the fact that any possibly limit space
would be with $g^+_\infty$ being hyperbolic. Therefore, by the proof of Theorem \ref{maintheorem1}, one concludes
that $(X, g_i^*)$ converges to the Fefferman-Graham's compactification of hyperbolic space in $C^{k,\alpha}$ Cheeger-Gromov
topology for some $\alpha\in (0, 1)$, from which we reached a contradiction.  
\end{proof}

Before ending this section, we turn to an important fact which is a consequence of the compactness result in Theorem \ref{maintheorem1},  and in fact is an 
improved statement of the compactness for conformally compact Einstein metrics with the same conformal infinity. We will later use this fact (Theorem \ref{weighted-conv}) 
to establish the uniqueness result in section 5. 
To state the result, we will use the notion of weighted
spaces of tensors, which we refer the readers to \cite{Lee1} (see also \cite{GL}). We first recall the following expansions in terms of the Fermi coordinate from
the boundary, which we have stated in
the introduction of this paper in terms of geodesic coordinates. The expansion in this form is motivated 
by an observation in \cite{Mald}.

\begin{lemm} \label{weighted-convbis} Let $(X^4, g)$ be a Bach-flat and Q-flat 4-manifold with the totally geodesic boundary. Then, in the Fermi coordinate from
the boundary, one has $g = dr^2 + g_r$ and the expansion
\begin{equation}\label{Mald-Expansion}
g_r = \hat g + g^{(2)}r^2 + g^{(3)}r^3 + \cdots,
\end{equation}
where $g^{(2)}$ is a curvature of $\hat g = g|_{T\p X}$ and $g^{(3)}$ is not local.
\end{lemm}

\begin{theo}\label{weighted-conv} 
Suppose that, for two sequences of conformally compact Einstein metrics $g_i^+$ and $h_i^+$ that have the same conformal infinity of
positive Yamabe type and satisfy the assumption  in Theorem \ref{maintheorem1} (or the assumptions in Theorem \ref{maintheorem3}). Then, for a weight $\delta \in (0, 3)$,
there are subsequence (possibly after diffeomorphisms $\psi_i$ and $\phi_i$ that fix the boundary) that  $\psi_i^*g_i^*-\phi_i^*h_i^*$ converges  in 
weighted $C^{2, \alpha}_\delta$ topology, where  $g_i^*$ (resp. $h_i^*$) denotes Fefferman-Graham's compactification of $g_i^+$ (resp. $h_i^+$). 
\end{theo}

\begin{proof} For each Fefferman-Graham's compactification $g_i^*$ (resp. $h_i^*$), we first set the Fermi coordinate from the boundary. By the lower
bound of the boundary injectivity radius, we know that the heights of these Fermi coordinates are bounded from the below. The necessary 
diffeomorphisms that fix the boundary one needs to use is to make sure that each of these Fefferman-Graham compactification $g_i^*$ 
share the same distance function $r$ to the boundary $\p X$ at least within the focal loci of $g_i^*$(resp. $h_i^*$).  
\\

Suppose that $g_i^*$ (resp. $h_i^*$) (a subsequence) converges in $C^{3, \alpha}$ Cheeger-Gromov topology.  Now, let us align the distance functions to
be the same for all $g_i^*$ (resp. $h_i^*$) in this subsequence by diffeomorphisms $\psi_i$ (resp. $\phi_i$) that fix the boundary, and get
$$
\psi_i^* g_i^* - \phi_i^* h_i^*= O(r^3). 
$$
for any $i$ from Lemma \ref{weighted-convbis}. If necessary, extract a subsequence,  for $\delta \in (0, 3)$ and any $\epsilon > 0$, 
there is an index $N$, for $i, j\ge N$,
$$
\|(\psi_i^*g_i^*- \phi_i^*h_i^*)- (\psi_j^*g_j^*- \phi_j^*h_j^*)\|_{C^{2, \alpha}(\bar X_\infty)} < \frac 12 \epsilon.
$$
For any fixed $\delta < 3$, one gets
$$
\|\psi_i^*g_i^* - \phi_i^*h_i^*\|_{C^{2, \alpha}} \leq Cr^3 \leq \epsilon r^\delta
$$
over the region $\{r\leq r_\epsilon\}$ for some small $r_{\epsilon}>0$ such that $Cr_\epsilon^{3- \delta} \leq \epsilon$ 
($C$ is independent of $i$ due to the compactness in Theorem \ref{maintheorem1} ( resp. Theorem \ref{maintheorem3}) and Lemma \ref{weighted-convbis}, and the sizes of 
Fermi coordinates for $g_i^*$ has a uniform lower bound again follows from Theorem \ref{maintheorem1} (resp. Theorem \ref{maintheorem3})); while one gets
$$
\|(\psi_i^*g_i^*- \phi_i^*h_i^*)- (\psi_j^*g_j^*- \phi_j^*h_j^*)\|_{C^{2, \alpha}(\bar X_\infty)}  \leq \epsilon r^\delta
$$
over the rest $\{r \ge r_\epsilon\}$ by setting $N$ larger, in the light of Theorem \ref{maintheorem1} (resp. Theorem \ref{maintheorem3}). It is then easily seen that the 
corresponding $\psi_i^*g_i^+ - \phi_i^*h_i^+$ converges in the weighted space 
$C^{2, \alpha}_\delta (X_\infty)$ with $\delta\in(0, 3)$. This completes the proof of Theorem \ref{weighted-conv}. 
\end{proof}


\section{Uniqueness of Graham-Lee solutions in dimension 4} \label{Sect:uniqueness}

In this section we derive the global uniqueness result Theorem \ref{uniqueness} based on some result in the recent work \cite{LQS}. 

\begin{proof}[Proof of Theorem \ref{uniqueness}] 
We will prove this by contradiction. Assume otherwise there is a sequence of conformal 3-sphere $(\mathbb{S}^3, [\hat g_i])$ 
that converges to the round sphere such that, for each $i$, there exist two non-isometric conformally compact Einstein metrics $g^+_i$
and $h^+_i$. 
\\

We first claim that, for a subsequence, both $g^+_i$ and $h^+_i$ converge to the hyperbolic space in $C^{3, \alpha}$ 
Cheeger-Gromov sense due to Theorem \ref{maintheorem2} and the uniqueness result when the conformal infinity is the standard sphere \cite{Q,LQS}.\\

Next, according to the proof of Theorem \ref{weighted-conv}, we actually can conclude that, after some diffeomorphisim that fix the 
boundary $\phi_i$ and $\psi_i$,
$$
\|\psi_i^* g^+_i - \phi_i^*h^+_i\|_{C^{2, \alpha}_\delta (\mathbb{H}^4)} \to 0
$$
for any $\delta \in (0, 3)$ and some subsequence. In other words, in this subsequence, the two distinct conformally
compact Einstein metrics $g^+_i$ and $h^+_i$ are arbitrarily close to each other in weighted spaces, as long $i$ is sufficiently
large. We will show this is not possible by applying the local uniqueness result via the implicit function theorem on weighted norm space established in \cite{GL, 
Lee1}. We now give more details.
\\

We denote $\psi_i^* g^+_i$ by $ g^+_i$ and $\phi_i^*h^+_i$ by $h^+_i$. We denote Fefferman-Graham's compactification $g_{i}^*=\rho_{i}^2 g_{i}^+$ and  $h_{i}^*=\tilde\rho_{i}^2 h_{i}^+$ where $\log\rho_{i}$ and $\log \tilde \rho_{i}$ solve (\ref{fg}).  Let us consider the nonlinear functional on $4$-dimensional ball $B^4$ introduced by Biquard \cite{Biquard1} for two metrics $g^+$ and $t^+$.  
\beq
 F(g^+,t^+):=Ric[g^+]+3g^+- \delta_{g^+}(B_{t^+}(g^+)),
\eeq
where $B_{t^+}(g^+)$ is a linear condition, essentially the infinitesimal version of the previous harmonicity condition
$$
B_{t^+}(g^+):=\delta_{t^+} g^++\frac12 d\tr_{t^+}(g^+).
$$
We have for any asymptotically hyperbolic Einstein metrics $h^+$
$$
D_1 F(h^+,h^+)=\frac12(\triangle_L+6),
$$
where $D_1$ denotes the differentiation of $F$ with respective to its first variable, and where the Lichnerowicz Laplacian $\triangle_L$ on symmetric $2$-tensors  is given by
$$
\triangle_L:=\nabla^*\nabla[g^+] +2\overset{\circ}{Ric}[g^+]-2\overset{\circ}{Rm}[g^+]; 
$$
where 
$$
\overset{\circ}{Ric}[g^+](u)_{ij}=\frac{1}{2}(R_{ik}[g^+]{u_j}^k+R_{jk}[g^+]{u_i}^k), 
$$
and
$$
\overset{\circ}{Rm}[g^+](u)_{ij}=R_{ikjl}[g^+]u^{kl}.
$$
It is clear
$$
 F(g_{i}^+, g_{i}^+)=0
$$
We now divide the proof in 2 steps.
\\

{\sl Step 1. } {\bf Claim.} {\it We could find a  diffeomorphism $\varphi_i$  of class $C^{3,\alpha}$ (equal to the identity on the boundary), such that 
$$
 F(\varphi_i^* h_{i}^+, g_{i}^+)=0
$$
Moreover $\|\varphi_i(x)-x\|_{C^{3,\alpha}}\to 0$ and $\|\varphi_i^* h_{i}^+-g_{i}^+\|_{C^{2,\alpha}_\delta}\to 0$ for some $\delta\in (2+\alpha,3)$.}
\vskip .1in
It is sufficient to check this infinitesimally: the diffeomorphism group acts infinitesimally on $g_i^+$ by taking the covector field $X_i$ to the symmetrized covariant derivative $(\delta_{g_i^+})^*X_i$, so the problem to solve is
$$
B_{g_{i}^+} ((\delta_{g_{i}^+})^*X_i) = -B_{g_{i}^+}(h_{i}^+).
$$
Recall
$$
B_{g_{i}^+} (\delta_{g_{i}^+})^* = \frac 12((\nabla )^*\nabla [g_{i}^+] -Ric[g_{i}^+] ),
$$
On the other hand, a direct calculation leads to  (Proposition 2.5 \cite{GL})
$$
B_{g^+}(h^+)=\rho^{-1}{\mathcal E}^0(h,g)+{\mathcal E}^1(h,g)
$$
where $h=r^2 h^+$ and $g=r^2 g^+$ for some defining function $r$, ${\mathcal E}^m$  denotes any tensor whose components in any coordinate system smooth up to the boundary of the $g,g^{-1}, h,h^{-1}$ and their partial derivatives such that in each term the total number of derivatives of $g$ and $h$ that appear is at most $m$. More precisely, we have
$$
\begin{array}{lll}
\ds {\mathcal E}^0(h,g)_m=-h^{jk}(g_{jk}r_m-4h_{mj}r_k),\\
\ds {\mathcal E}^1(h,g)_m=-h^{jk}\p_kg_{mj}-\Gamma(h)^l_{mk}g_{lj}-\Gamma(h)^l_{jk}g_{lm}+\frac12\p_m(h^{jk}g_{jk})
\end{array}
$$
If there is no confusion, we drop the index $i$ for the metrics $g_i^+, h_{i}^+, g_{i}^*, h_{i}^*$, the covector field $X_i$.  In view of Theorem \ref{weighted-conv}, we note $B_{g^+}(h^+)=B_{g^+}(h^+-g^+)\in C^{1,\alpha}_{\delta}$ for all $\delta\in (0,3)$. Moreover, $B_{g^+_i}(h^+_i)\to 0$ in $C^{1,\alpha}_{\delta}$. We consider a $C^1$ fully nonlinear operator $\Psi$ for $\delta\in (2,3)$
$$
\begin{array}{lllll}
\Psi:& {C^{3,\alpha}_{\delta}( B^4;TB^4)}&\to&   {C^{1,\alpha}_{\delta}( B^4;T B^4)} \\
&\tilde X&\mapsto& \widetilde{B_{g^+}(exp(\tilde X)^*h^+)}
\end{array}
$$
where $exp$ is the exponential map and $\tilde B$ is a vector field related to the one form $B$. We know 
$2d\Psi(0)=\nabla^*\nabla +3$. It follows from Theorem C \cite{Lee1} that $d\Psi(0): C^{3,\alpha}_{\delta}(B^4;T B^4)\to C^{1,\alpha}_{\delta}(B^4;T B^4)$ is an isomorphism provided $2<\delta<3$.  Applying inverse functions theorem, for large $i$, we find $\tilde X_i\in C^{3,\alpha}_{\delta}$ such that 
$$
\Psi(\tilde X_i)=0.
$$
Again from Lemma 3.7 \cite{Lee1}, we have $C^{3,\alpha}_{\delta} ( B^4;T B^4) \subset C^{3,\alpha}_{2+\alpha} ( B^4;T B^4)\subset C^{3,\alpha}_{(0)} ( \bar B^4;T \bar B^4)$ provided $\delta>2+\alpha$. Thus, we  find a  diffeomorphism $\varphi_i=\exp(\tilde X_i)$  of class $C^{3,\alpha}$ (equal to the identity on the boundary), such that 
$$
 F(({\varphi_i})_*h_{i}^+, g_{i}^+)=0. 
$$ 
Moreover,  $\|\varphi_i(x)-x\|_{C^{3,\alpha}}\to 0$ and $\|\varphi_i^* h_{i}^+-g_{i}^+\|_{C^{2,\alpha}_\delta}\to 0$ for $\delta\in (2+\alpha,3)$. Thus we proved the claim.
\\

{\sl Step 2. } {\bf Claim.} {\it For large $i$, we have 
$$
g_{i}^+=\varphi_i^*h_{i}^+.
$$
}\\
We know $ F(\varphi_i^*h_{i}^+,g_{i}^+)=  F(g_{i}^+,g_{i}^+)=0$ and by step 1 $\|g_{i}^+-\varphi_i^*h_{i}^+\|_{C^{2,\alpha}_\delta}\to 0$ for $\delta\in (2+\alpha,3)$. On the other hand, using \cite[Theorems C and D]{Lee1}  and \cite[Lemma 12.71]{Besse}, the linearized operator $$D_1  F( g_{i}^+, g_{i}^+):  C^{2,\alpha}_{\delta}(  B^4;\Sigma^2   B^4)\to  C^{0,\alpha}_{\delta}(  B^4;\Sigma^2   B^4)$$ is an isomorphism. Applying the implicit function theorem, we infer the claim.

\end{proof}

\end{document}